\documentclass[10pt,journal,cspaper,onecolumn,compsoc]{IEEEtran}
\usepackage{tikz}

\usepackage{amsthm}

\usepackage{graphicx}

\ifCLASSOPTIONcompsoc
\else
\fi
\usepackage[cmex10]{amsmath}
\usepackage{array}

\usepackage{mdwmath}
\usepackage{mdwtab}

\usepackage{eqparbox}
\usepackage{fixltx2e}

\usepackage{stfloats}

\usepackage{caption}
\usepackage{subcaption}

\newtheorem{theorem}{Theorem}
\newtheorem{lemma}{Lemma}
\newtheorem{definition}{Definition}
\newtheorem{openproblem*}{Open Problem}
\newtheorem{corollary}{Corollary}
\begin{document}
%
\title{Combinatorics of $k$-Interval Cospeciation for Cophylogeny}
%
%
%
%

\author{Jane Ivy Coons, Joseph Rusinko
\IEEEcompsocitemizethanks{\IEEEcompsocthanksitem J. Coons is a student at the State University of New York at Geneseo and a participant in the Winthrop University Research Experience for Undergraduates.\protect\\
E-mail: janeivycoons@gmail.com
\IEEEcompsocthanksitem J. Rusinko is with Winthrop University.
Email: rusinkoj@winthrop.edu
}
\thanks{}}

\IEEEcompsoctitleabstractindextext{%
\begin{abstract}
We show that the cophylogenetic distance, $k$-interval cospeciation, is distinct from other metrics and accounts for global congruence between locally incongruent trees. The growth of the neighborhood of trees which satisfy the largest possible k-interval cospeciation with a given tree indicates that $k$-interval cospeciation  is useful for analyzing simulated data.
\end{abstract}

}

\maketitle

\IEEEdisplaynotcompsoctitleabstractindextext

%
\IEEEpeerreviewmaketitle

\section{Introduction}
%
%

%
%
%
%
\IEEEPARstart{T}{he} field of phylogenetics aims to discern evolutionary relationships between species using biological data and mathematics. A relatively new subfield of phylogenetics, known as \emph{cophylogenetics}, concerns the evolutionary processes of groups of taxa, or taxonomical units of any rank, that we believe are evolving concomitantly. Cophylogenetics is applicable, for instance,  in the cases of host-parasite coevolution and the coevolution of a species and the genes within that species \protect\cite{huggins2012first}, \protect\cite{hughes2007multiple}.

In cophylogenetics, we require tools such as reconstruction algorithms and distance metrics to determine the evolutionary processes of the taxa we are studying and how these processes are related. There are many reasons why host-parisite and gene-species trees do not match exactly, such as a parasite's failure to speciate in reaction to a host speciation, parasite extinction, or independent speciation of the parasite \protect\cite{huggins2012first}; however, there are expected similarities between the trees. Therefore, $k$-interval cospeciation is defined in order to measure these differences between trees, and quantify their level of congruence \protect\cite{huggins2012first}. 

This is applicable to cophylogeny when, for instance, $T_1$ represents a host tree, $T_2$ represents a parasite tree, and a pair of taxa, $A$ and $a$ are associated if $a$ is a parasite that lives on a host, $A$. Two $n$-taxa trees, $T_1$ and $T_2$, satisfy \emph{$k$-interval cospeciation} if for all pairs of corresponding taxa on the trees, the number of edges between two taxa, $A$ and $B$, on $T_1$ is within $k$ of the number of edges between their corresponding taxa, $a$ and $b$, on $T_2$.   We can also apply this to two possible trees on the same set of taxa -- such as several gene trees -- where two taxa on different trees are associated if they belong to the same species. In \protect\cite{huggins2012first}, the authors pose an open problem to characterize the combinatorial properties of $k$-interval cospeciation; this paper is a response to that open problem.

Several studies of host-parasite relationships and of gene-species trees compare cophylogenetic trees in order to understand coevolution \protect\cite{hughes2007multiple}, \protect\cite{refregier2008cophylogeny}, \protect\cite{hafner1988phylogenetic} \protect\cite{maddison1997gene}. In \protect\cite{refregier2008cophylogeny}, the authors constructed phylogenetic trees for Caryophylaceae and Microbotryum (hosts and parasites, respectively) using a Bayseian 50\% majority rule consensus tree approach. They noted that on a large scale, the trees appear quite congruent, though they have many local incongruences. They were able to determine this through visual inspection, as their trees were relatively small; however, for larger trees, this may prove difficult. Our research indicates that $k$-interval cospeciation is a viable means of quanitifying the degree of this global congruence. For instance, the host-parasite trees in Figure 4 of  \protect\cite{hughes2007multiple}, which describe the evolutionary processes of pelicaniform birds and the lice that live on them, satisfy 5-interval cospeciation. This relatively low $k$-interval cospeciation accounts for the high level of global congruence observed between these trees. 

Distance metrics such as $k$-interval cospeciation are also useful for analyzing how phylogenetic methods perform on simulated data; however, some metrics are more suited to this task than others because of their combinatorial properties. We show that, because of the growth of the neighborhood of trees that satisfy the largest possible $k$-interval cospeciation with a given tree, $k$-interval cospeciation is more useful than other common distance metrics for analyzing data generated from simulations.

\section{Background}

An \emph{unrooted binary tree} is a phylogenetic tree where each non-leaf vertex has degree three. At each leaf of the tree is a \emph{taxon} (plural: taxa), such as a gene, species or population. We define a \emph{cherry} to be a pair of taxa with exactly two edges between them, and a \emph{caterpillar tree} to be a phylogenetic tree with exactly two cherries. Since $k$-interval cospeciation is a discrete distance metric, we do not consider branch lengths.

\begin{definition}
Let $\varepsilon_H(A,B)$ be the number of edges in the shortest path between taxa $A$ and $B$ on $T_H$ and $\varepsilon_P(a,b)$ be the number of edges in the shortest pathbetween taxa $a$ and $b$ on $T_P$. Then $T_H$ and $T_P$ satisfy \textbf{$k$-interval cospeciation} ($k$-IC) if the following inequality holds for all $A$ and $B$, and corresponding $a$ and $b$:
\begin{equation} \label{eq:kICinequality}
\vert \varepsilon_H(A,B) - \varepsilon_P(a,b) \vert \leq k.
\end{equation}
\end{definition}

Note that, by (\ref{eq:kICinequality}), if two trees satisfy $k$-IC, they also satisfy $(k+x)$-IC for any positive integer, $x$. So, we say that two trees satisfy \emph{precise} $k$-IC if they satisfy $k$-IC and they do not satisfy $(k-1)$-IC. If $T_1$ and $T_2$ satisfy precise $k$-IC, then $d(T_1,T_2) = k$.

\begin{theorem}
Precise $k$-IC is a tree metric.
\end{theorem}

\begin{proof}
For any tree, $d(T,T)=0$, since the number edges between each pair of taxa must remain the same. Additionally, for any two $n$-taxa trees, $T_1$ and $T_2$, $d(T_1,T_2)=d(T_2,T_1)$. So, in order for precise $k$-IC to be a distance, we must show that it satisfies the triangle inequality.

Let $T_1$, $T_2$, $T_3$ be $n$-taxa trees for which $d(T_1,T_2)=x$, $d(T_2,T_3)=y$, and $d(T_1,T_3)=z$ for some $x,y,z \in \textbf{Z}_{\geq 0}$. We must show that $z \leq x+y$. Consider arbitary leaves, $i_1, j_1 \in T_1$, $i_2, j_2 \in T_2$ and $i_3, j_3 \in T_3$. By definition of $k$-IC, 
\begin{eqnarray}
-x & \leq & \varepsilon_1(i_1,j_1) - \varepsilon_2(i_2,j_2)  \leq   x, \label{T1T2exp} \\
-y & \leq & \varepsilon_3(i_3,j_3) - \varepsilon_2(i_2,j_2)   \leq   y, \label{T2T3exp} \\
-z & \leq & \varepsilon_1(i_1,j_1) - \varepsilon_3(i_3,j_3) \leq z. \label{T1T3exp}
\end{eqnarray}
Adding \ref{T1T2exp} and \ref{T2T3exp} yields $-x-y \leq (\varepsilon_1(i_1,j_1) - \varepsilon_2(i_2,j_2)) - (\varepsilon_2(i_2,j_2) - \varepsilon_3(i_3,j_3)) \leq x + y$, which simplifies to $\vert  \varepsilon_1(i_1,j_1) - \varepsilon_3(i_3,j_3) \vert  \leq  x + y.$ So, by definition of  $k$-IC, $T_1$ and $T_3$ satisfy $(x+y)$-IC.
So, in order for $T_1$ and $T_3$ to precisely satisfy $z$-IC, $z \leq x+y$, as needed. Therefore, precise $k$-IC satisfies the triangle inequality. So, precise $k$-IC is a distance.
\end{proof}

\section{Precise $k$-Interval Cospeciation and Other Distances}
A comparison of precise $k$-interval cospeciation to other common discrete tree distances indicates that $k$-IC is a novel metric. We say that a distance metric, $d_a$, \emph{determines} another distance, $d_b$ if there exists a bijection, $f$ such that $f(d_a(T_1,T_2))=d_b(T_1,T_2)$ for all $T_1$ and $T_2$. We found that several of the most commonly used discrete distance metrics do not determine $k$-IC, and that $k$-IC reflects global similarities between trees which other distance metrics do not. This indicates that $k$-IC is a unique metric with properties that are not represented by current phylogenetic distances.

One extremely common distance metric is the nearest neighbor interchange distance. A \emph{nearest neighbor interchange} (NNI) is a tree rearrangement operation in which we switch two subtrees of a tree that are joined by a single edge. So, the nearest neighbor interchange distance between $T_1$ and $T_2$ is the smallest number of NNIs it takes to transform $T_1$ into $T_2$, or vice versa. Finding the NNI distance between two given trees is an NP-complete problem \cite{li1996nearest}.

\begin{theorem}
Determing the $k$-IC between two trees is computable in polynomial time.
\end{theorem}
\begin{proof}
 This is evident from the fact that that the shortest path between two nodes of a tree can be computed at $O(nlogn)$ \protect\cite{skiena1990dijkstra}, and from the fact that we must compute these path $\binom{n}{2}$ times. So, computing $k$-IC can be done at $O(n^3logn)$.
\end{proof}
 
For this reason, we hoped to find a relationship between the NNI distance and $k$-IC in order to bound the NNI distance more efficiently \protect\cite{huggins2012first}.

Huggins et al. proved that two trees satisfy 1-IC if and only if their NNI distance is one \protect\cite{huggins2012first}. However, they also introduced a counterexample to the possibility of $k$-IC providing an upper bound for the NNI distance for any $k>1$. These counterexample trees, $C_{3x}$ and $D_{3x}$, can be built by adding rooted triples on the same leaf set, but with differing cherries to the pendant edges of two $x$-taxa caterpillar trees, $C$ and $D$. Fig.~\ref{fig:CounterTree1} and Fig.~\ref{fig:CounterTree2} are examples of these trees on 15 taxa.

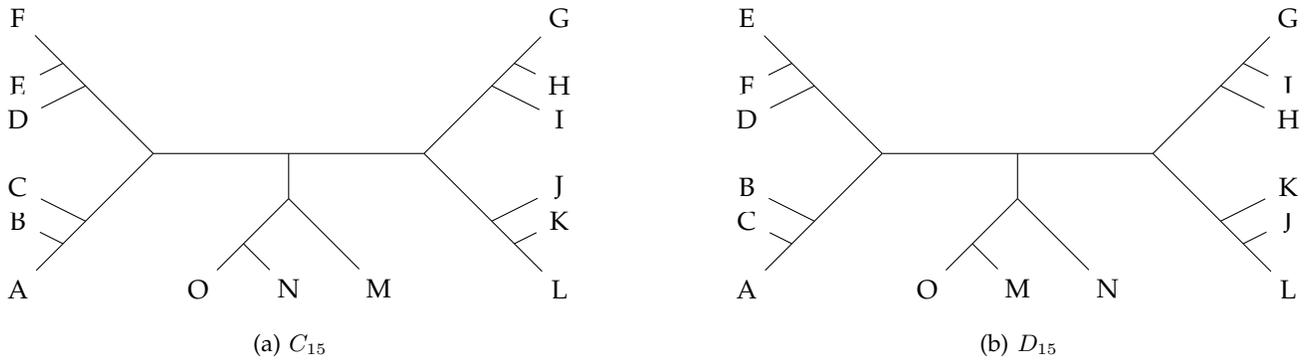
\begin{figure}
\centering
	\begin{subfigure}{.45\textwidth}
		\begin{tikzpicture}[scale=.6]
\draw (3,3) -- (0,6) node[circle, fill=white](nodeF) {F};
\draw(1,5) -- (0, 4.5) node[circle, fill=white]{E};
\draw(1.5,4.5) -- (0,3.75) node[circle,fill=white]{D};
\draw (3,3) -- (0,0) node[circle, fill=white]{A};
\draw(1,1) -- (0,1.5) node[circle, fill=white]{B};
\draw(1.5,1.5) -- (0, 2.25) node[circle, fill=white]{C};
\draw(3,3) -- (9,3);
\draw(6,3) -- (6,2);
\draw(6,2) -- (4,0) node[circle, fill=white]{O};
\draw(5,1) -- (6,0) node[circle, fill=white]{N};
\draw(6,2) -- (8,0) node[circle, fill=white]{M};
\draw(9,3) -- (12,6) node[circle, fill=white]{G};
\draw(11,5) -- (12, 4.5) node[circle, fill=white]{H};
\draw(10.5,4.5) -- (12, 3.75) node[circle, fill=white]{I};
\draw(9,3) -- (12,0)  node[circle, fill=white]{L};
\draw(11,1) -- (12, 1.5) node[circle, fill=white]{K};
\draw(10.5,1.5) -- (12, 2.25) node[circle, fill=white]{J};
\end{tikzpicture}
		\caption{$C_{15}$ \label{fig:CounterTree1}}
		
\end{subfigure} 
\qquad \qquad 
\begin{subfigure}{.45\textwidth}
	\begin{tikzpicture}[scale=.6]
\draw (3,3) -- (0,6) node[circle, fill=white](nodeF) {E};
\draw(1,5) -- (0, 4.5) node[circle, fill=white]{F};
\draw(1.5,4.5) -- (0,3.75) node[circle,fill=white]{D};
\draw (3,3) -- (0,0) node[circle, fill=white]{A};
\draw(1,1) -- (0,1.5) node[circle, fill=white]{C};
\draw(1.5,1.5) -- (0, 2.25) node[circle, fill=white]{B};
\draw(3,3) -- (9,3);
\draw(6,3) -- (6,2);
\draw(6,2) -- (4,0) node[circle, fill=white]{O};
\draw(5,1) -- (6,0) node[circle, fill=white]{M};
\draw(6,2) -- (8,0) node[circle, fill=white]{N};
\draw(9,3) -- (12,6) node[circle, fill=white]{G};
\draw(11,5) -- (12, 4.5) node[circle, fill=white]{I};
\draw(10.5,4.5) -- (12, 3.75) node[circle, fill=white]{H};
\draw(9,3) -- (12,0)  node[circle, fill=white]{L};
\draw(11,1) -- (12, 1.5) node[circle, fill=white]{J};
\draw(10.5,1.5) -- (12, 2.25) node[circle, fill=white]{K};
\end{tikzpicture}		
\caption{$D_{15}$ \label{fig:CounterTree2}}
\end{subfigure}
\caption{Counterexample Trees Introduced in \protect\cite{huggins2012first}. \label{fig:CounterTrees}} 
\end{figure}

Since no NNI operation is able to alter two of these rooted triples at the same time, the NNI distance between $C_{3x}$ and $D_{3x}$ is $x$. However, the two trees satisfy 2-IC. So, $k$-IC does not provide an upper bound for the NNI distance between trees. Therefore, the NNI distance does not determine $k$-IC. However, $k$-IC does provide a strict lower bound on the NNI distance.

\begin{theorem} Suppose two phylogenetic trees, $T_1$ and $T_2$ precisely satisfy $k$-IC. Then $d_{NNI}(T_1,T_2) \geq k$.
\end{theorem}

\begin{proof}
It suffices to show that one NNI operation changes the number of edges in the shortest path between any two taxa on a tree by at most one. Consider the trees, $T_1$ and $T_i$ in Fig.~\ref{fig:NNIPics}, which are reduced to a single, arbitrarily chosen internal edge with two subtrees extending from each of its vertices. The subtrees are denoted by $S_1$, $S_2$, $S_3$ and $S_4$. Note that these are subtrees and not necessarily external edges, as indicated by the arrows at the end of each edge. For $n > 3$, all unrooted, binary trees with $n$-taxa can be reduced to a single internal edge with four subtrees extending from it, as in Fig.~\ref{fig:NNIPic1} and Fig.~\ref{fig:NNIPic2}. Consider the NNI that switches $S_1$ and $S_2$ on $T_1$ without loss of generality. Every possible NNI on $T_1$ and $T_i$ will take this form, as an NNI switches two subtrees that are connected by a single internal edge. This results in the tree, $T_i$.

\begin{figure}
\centering
\begin{subfigure}{.45\textwidth}
\centering
\begin{tikzpicture}[scale=.6]
\draw[->] (2,2) -- (1,3);
\node at (0.5,3.5) {$S_1$};
\draw[->] (2,2) -- (1,1);
\node at (0.5,0.5) {$S_3$};
\draw (2,2) -- (4,2);
\draw[->](4,2) -- (5,3);
\node at (5.5,3.5) {$S_2$};
\draw[->](4,2) -- (5,1);
\node at (5.5,0.5) {$S_4$};
\end{tikzpicture}
\caption{$T_1$}
\label{fig:NNIPic1}
\end{subfigure}
\begin{subfigure}{.45\textwidth}
\centering
\begin{tikzpicture}[scale=.6]
\draw[->] (2,2) -- (1,3);
\node at (0.5,3.5) {$S_2$};
\draw[->] (2,2) -- (1,1);
\node at (0.5,0.5) {$S_3$};
\draw (2,2) -- (4,2);
\draw[->](4,2) -- (5,3);
\node at (5.5,3.5) {$S_1$};
\draw[->](4,2) -- (5,1);
\node at (5.5,0.5) {$S_4$};
\end{tikzpicture}
\caption{$T_i$}
\label{fig:NNIPic2}
\end{subfigure}
\caption{\label{fig:NNIPics}}
\end{figure}

We know that the topology of a subtree on $T_1$ cannot be changed by an NNI about this internal edge. An NNI operation can insert or delete at most one edge between any two taxa. Then, since $T_i$ is one NNI operation away from $T_1$, $k-1 \leq d(T_i,T_2) \leq k+1$. Therefore, it takes at least $k$ NNI operations in order to transform $T_1$ into $T_2$. So, the NNI distance between $T_1$ and $T_2$ is at least $k$.
\end{proof}

The lower bound is strict, since we can always find two $n$-taxa trees which satisfy $k$-IC and which can be rearranged to match in a series of $k$ NNI moves. For example, in Fig.~\ref{fig:LowerBound}, we have with two $n$-taxa caterpillar trees, where $n = k+4$. We can get from $T_1$ to $T_2$ in a series of $k$ NNI operations by simply switching $i_3$ on $T_1$ with the pendant edge to the left of it $k$ times. Since the NNI distance between $T_1$ and $T_2$ cannot be less than $k$, their NNI distance is $k$. Though $k$-IC may not provide the best possible lower bound on the NNI distance, it is much quicker to compute than the NNI distance itself.

\begin{figure}
\centering
\begin{subfigure}{.45\textwidth}
\centering
\begin{tikzpicture}[scale = .55]
\draw (2,2) -- (0,4) node[circle, fill=white]{$i_1$};
\draw (2,2) -- (0,0) node[circle, fill=white]{$i_2$};
\draw (2,2) -- (5.5,2);
\draw (3,2) -- (3,0) node[circle, fill=white]{$i_3$};
\draw (4,2) -- (4,4) node[circle, fill=white]{$i_4$};
\draw (5,2) -- (5,0) node[circle, fill=white]{$i_5$};
\fill (6,2) circle (3pt);
\fill (6.5,2) circle (3pt);
\fill (7,2) circle (3pt);
\draw (7.5,2) -- (11,2);
\draw (8,2) -- (8,0) node[circle, fill=white]{$i_{n-4}$};
\draw (9,2) -- (9,4) node[circle, fill=white]{$i_{n-3}$};
\draw (10,2) -- (10,0) node[circle, fill=white]{$i_{n-2}$};
\draw (11,2) -- (13,0) node[circle, fill=white]{$i_{n-1}$};
\draw (11,2) -- (13,4) node[circle, fill=white]{$i_{n}$};
\end{tikzpicture}
\caption{$T_1$ \label{fig:LowerBound1}}
\end{subfigure}
\begin{subfigure}{.45\textwidth}
\centering
\begin{tikzpicture}[scale = .55]
\draw (2,2) -- (0,4) node[circle, fill=white]{$j_1$};
\draw (2,2) -- (0,0) node[circle, fill=white]{$j_2$};
\draw (2,2) -- (5.5,2);
\draw (3,2) -- (3,0) node[circle, fill=white]{$j_4$};
\draw (4,2) -- (4,4) node[circle, fill=white]{$j_5$};
\draw (5,2) -- (5,0) node[circle, fill=white]{$j_6$};
\fill (6,2) circle (3pt);
\fill (6.5,2) circle (3pt);
\fill (7,2) circle (3pt);
\draw (7.5,2) -- (11,2);
\draw (8,2) -- (8,0) node[circle, fill=white]{$j_{n-3}$};
\draw (9,2) -- (9,4) node[circle, fill=white]{$j_{n-2}$};
\draw (10,2) -- (10,0) node[circle, fill=white]{$j_{3}$};
\draw (11,2) -- (13,0) node[circle, fill=white]{$j_{n-1}$};
\draw (11,2) -- (13,4) node[circle, fill=white]{$j_{n}$};
\end{tikzpicture}
\caption{$T_2$ \label{fig:LowerBound2}}
\end{subfigure}
\caption{\label{fig:LowerBound}}
\end{figure}

Using similar methods with $C_{3x}$ and $D_{3x}$, we can show that $k$-IC does not determine the Robinson-Foulds distance, the path difference distance or the maximum agreement subtree distance. From these comparisons, we see that the $k$-IC between two trees describes global similarities between the trees that other common discrete distance metrics do not. This led us to believe that $k$-IC is a useful tool for studying pairs of cophylogenetic trees. 

One way of describing this global similarity is through the length of the longest path between two taxa on the trees. We use the longest paths on two trees that satisfy $k$-IC to quanitfy the observation that $k$-IC describes global similarities between trees. If the difference between the length of the longest paths on two trees is equal to $x$, then $k$ must be at least $x$.

\begin{theorem}
If two $n$-taxa trees, $T_1$ and $T_2$ precisely satisfy $k$-IC, then the absolute value of the difference between the length of the longest path on $T_1$ and the length of the longest path on $T_2$ is at most $k$. Additionally, if this difference is equal to $k$, then taxa on opposite ends of the longest path on $T_1$ must also be at the ends of the longest simple path on $T_2$.
\end{theorem}
\begin{proof}
Without loss of generality, let the longest path on $T_1$ be greater than the longest path on $T_2$. Let $A,B$ be the taxa which lie at the ends of the longest path on $T_1$. In other words, $\varepsilon_1(A,B) = \text{max}(\varepsilon_1(Q,R))$ for all $Q,R \in T_1$. Since $a,b \in T_2$ can be no farther apart than $\text{max}(\varepsilon_2(s,t))$, $\varepsilon_1(A,B) - \varepsilon_2(a,b) \geq \text{max}(\varepsilon_1(Q,R))-\text{max}(\varepsilon_2(s,t))$. In other words, the difference between the number of edges between $A$ and $B$ and the number of edges between $a$ and $b$ must be greater than the difference between the longest paths on $T_1$ and $T_2$. So, $\vert \text{max}(\varepsilon_1(Q,R))-\text{max}(\varepsilon_2(s,t)) \vert \leq k$. If $\vert \text{max}(\varepsilon_1(Q,R))-\text{max}(\varepsilon_2(s,t)) \vert = k$, then $a$ and $b$ must lie along the longest path in $T_2$. If they did not, $\varepsilon_1(A,B) - \varepsilon_2(a,b)$ would be greater than $k$, contradicting the fact that $d(T_1,T_2)=k$.
\end{proof}
This result helps to quanitify the degree of global congruence between two cophylogenetic trees so that researchers need not rely upon visual inspection.

\section{The $(n-3)$-Interval Neighborhood}

The Robinson-Foulds Distance is one of the most commonly used distance metrics in phylogenetics. However, as papers such as \cite{snir2010quartets} mention, the growth of the neighborhood of trees the farthest possible Robinson-Foulds distance away from a given tree can make it difficult to use for analysis of simulated data. In order to analyze these neighborhoods in terms of $k$-IC, we define the $k$-interval neighborhood. The \emph{$k$-interval neighborhood} ($\text{IN}_k$) of a tree, $T$, is defined to be the set of all trees which precisely satisfy $k$-IC with $T$. By analyzing the size of the neighborhood of trees which precisely satisfy the largest possible $k$-IC with a given tree, we can gain insight into how $k$-IC can be used to analyze simulated data.

\subsection{Size of the $(n-3)$-Interval Neighborhood}

\begin{lemma} \label{lemma:n-3IC}
Two $n$-taxa trees precisely satisfy at most $(n-3)$-IC.
\end{lemma}

\begin{proof}
Consider two $n$-taxa trees, $T_1$ and $T_2$. Then $T_1$ and $T_2$ each have exactly $n-3$ internal edges \protect\cite{robinson1971comparison}. Therefore, two taxa can have at most $n-1$ edges between them, and there must be at least 2 edges between them. So, to maximize the $k$-interval cospeciation of $T_1$ and $T_2$, let $\varepsilon_1(I,J) = n-1$ for some $I,J \in T_1$ and $\varepsilon_2(i,j) = 2$ for corresponding $i,j \in T_2$. Then $\varepsilon_1(I,J) - \varepsilon_2(i,j) = n-3$. Since this method maximizes $k$, $d(T_1,T_2) \leq (n-3)$.
\end{proof}

\begin{corollary}\label{lemma:n-3Cat}
If two for $n$-taxa trees, $T$ and $T'$, $d(T,T')=n-3$, at least one of them must be a caterpillar tree.
\end{corollary}

This follows directly from the fact that, in order for two trees to satisfy $(n-3)$-IC, one of those trees must have a path of length $n-1$. The only trees with such paths are caterpillar trees.

Let $ub(n)$ denote the number of unrooted binary trees on $n$ taxa. We know from \protect\cite{steel2014tracing} that 
\begin{equation*}
ub(n) = (2n-5)!! = (2n-5) \times (2n-3) \times ... \times 5 \times 3 \times 1.
\end{equation*}

\begin{theorem}
For an $n$-taxa tree, $T$, with $c$ cherries,

\begin{equation*}
\vert \text{IN}_{(n-3)} \vert=
	\begin{cases}
	2((n-2)!+2((n-4)! - 2(n-3)!) +2(2n-7)!! - (2n-9)!!) & \text{if } c = 2, \\
	 c((n-2)! - 2(n-4)!(c-1)) & \text{if } c \geq 3.
	 \end{cases}
\end{equation*}
\end{theorem}

\begin{proof}
Case 1: $c = 2$.
When $c=2$, $T$ is a caterpillar tree. So, if $d(T,T')=n-3$, $T'$ can have any topology. Suppose $T$ has two cherries, $(a,b)$ and $(c,d)$. We can build such a $T'$ from $T$ either by:

\begin{itemize}
\item placing two taxa on opposite cherries of $T$ into the same cherry on $T'$, or
\item placing two taxa from the same cherry of $T$ into opposite ends of $T'$, when $T'$ is a caterpillar.
\end{itemize}

First, we count the trees obtained by placing two taxa on opposite cherries of $T$ into the same cherries on $T'$. There are four possible cherries to pick from: $(a,c)$, $(a,d)$, $(b,c)$ and $(b,d)$. Since the number of $n$-taxa unrooted binary trees with a single cherry fixed is equal to the number of $(n-1)$-taxa unrooted binary trees, fixing these four cherries yields $4ub(n-1)$ trees. However, it is possible for cherries $(a,c)$ and $(b,d)$ to appear on the same tree and for cherries $(a,d)$ and $(b,c)$ to appear on the same tree. Since the number of $n$-taxa unrooted binary trees with two cherries fixed is equal to the number of $(n-2)$-taxa unrooted binary trees, there are $2ub(n-2)$ trees which have been double-counted. So, $2(2ub(n-1)-ub(n-2)) = 4(2(2n-7)!!-(2n-9)!!)$ is the number of trees which precisely satisfy $(n-3)$-IC with $T$ yielded by putting taxa from opposite cherries in $T$ into the same cherry in $T'$.

The rest of the trees in $\text{IN}_{(n-3)}(T)$ can be counted by fixing two taxa that are in cherries in $T$ in opposite cherries of a caterpillar tree $T'$. There are $(n-2)!$ caterpillar trees with $a$ and $b$ fixed this way. However, $(n-3)!$ of those contain cherry $(a,c)$ and $(n-3)!$ contain cherry $(b,d)$, which we counted in the previous step. Of those $2(n-3)!$, $(n-4)!$ include \emph{both} $(a,c)$ and $(b,d)$ as cherries, so not all of them are unique. Therefore, there are $2(n-3)!-(n-4)!$ trees with cherry $(a,c)$ or $(b,d)$ which have already been counted. Similarly, there are $2(n-3)!-(n-4)!$ trees with cherry $(a,d)$ or $(b,c)$ which have already been counted. So, by fixing $a$ and $b$ in opposite cherries of a caterpillar tree, we yield $(n-2)!-2(2(n-3)!-(n-4)!)$ new trees in $\text{IN}_{(n-3)}(T)$.

By the same argument, we obtain another $(n-2)!-2(2(n-3)!-(n-4)!)$ new trees by fixing $c$ and $d$ in opposite cherries with no double counting. Therefore, for a caterpillar tree,

\begin{eqnarray}
\vert \text{IN}_{(n-3)}(T) \vert &=& 2((n-2)!-2(2(n-3)!-(n-4)!)) + 2(2ub(n-1)-ub(n-2) \nonumber \\
&=& 2((n-2)!-2(2(n-3)!-(n-4)! +2(2(2n-7)!! - (2n-9)!!))). \label{eqn:cat_n-3-IN}
\end{eqnarray} 

Case 2: $c \geq 3$.
When a tree, $T$, has $c$ cherries with $c \geq 3$, we know that $T$ cannot be a caterpillar. Therefore, in order for $d(T,T')=n-3$ to hold, $T'$ must be a caterpillar. Additionally, since a caterpillar is the only tree topology with an internal path length of $n-1$, we know that for each possible $T'$, there must be two taxa that were in a cherry on $T$ which are in opposite cherries on $T'$. 

For each cherry, $(i_l, j_l)$ on $T$, $1 \leq l \leq c$, there exist $(n-2)!$ caterpillar trees with $i_l$ and $j_l$ in opposite cherries. For $i_1$ and $j_1$, this amounts to $(n-2)!$ trees counted in $\text{IN}_{(n-3)}(T)$. For $(i_2, j_2)$, there are $(n-2)!$ caterpillar trees with $i_2$ and $j_2$ in opposite cherries. However, $(n-4)!$ have of these have cherries $(i_1, i_2)$ and $(j_1,j_2)$ and another $(n-4)!$ have cherries $(i_1,j_2)$ and $(j_1,i_2)$. So of the $(n-2)!$ caterpillar trees with $i_2$ and $j_2$ in opposite cherries, $2(n-4)!$ were already counted. In fact, for each $(i_l, j_l)$, $2(l-1)(n-4)!$ have already been counted for the previous cherries of $T$. So

\begin{eqnarray}
\vert \text{IN}_{(n-3)}(T) \vert &=& \sum^{c}_{l=1}((n-2)!-2(l-1)(n-4)!) \nonumber \\
 &=& c(n-2)! - 2(n-4)!\sum^{c}_{l=2}(l-1) \nonumber \\
 &=& c(n-2)! - 2(n-4)!\sum^{c-1}_{l=1}l \nonumber \\
 &=& c(n-2)! - \frac{2(n-4)!c(c-1)}{2} \nonumber \\
 &=& c((n-2)! - (n-4)!(c-1)). \label{eqn:n-3IN_clarge}
\end{eqnarray}
\end{proof}

\subsection{Growth of the $(n-3)$-Interval Neighborhood}
The size of the $\text{IN}_{(n-3)}$ is not constant for every $n$-taxa tree, $T$. It grows with respect to $n$ as well as with respect to the number of cherries, $c$, on $T$. Understanding how $\text{IN}_{(n-3)}$ grows with respect to tree topology can help us to understand what information $k$-interval cospeciation conveys about the two trees we are comparing.

\begin{theorem}
For all $n$-taxa trees with three or more cherries, a tree with three cherries has the smallest $(n-3)$-interval neighborhood, and a tree with $\lfloor \frac{n}{2} \rfloor$ cherries has the largest. The size of the $(n-3)$-interval neighborhood increases as a quadratic on $c$.
\end{theorem}
\begin{proof}
For $c \geq 3$, we have (\ref{eqn:n-3IN_clarge}). We can expand this formula in order to observe it as a quadratic function on $c$.
\begin{eqnarray}
\vert \text{IN}_{(n-3)} \vert &=& c(n-2)!-c(c-1)(n-4)! \nonumber \\
&=& c(n-2)!-c^2(n-4)!+c(n-4)! \nonumber \\
&=& -c^2(n-4)!+c((n-2)!+(n-4)!) \label{eqn:n-3IN_clarge_quadratic}
\end{eqnarray}
So, as $c$ increases, the size of $\text{IN}_{(n-3)}(T)$ is increasing as a concave down quadratic. The vertex of this parabola is located at $c = \frac{(n-2)! + (n-4)!}{2(n-4)!}$. So, if $c$ could ever exceed that value, the size of $\text{IN}_{(n-3)}(T)$ would start to decrease. In order to prove that this cannot happen, we show that $\lfloor \frac{n}{2} \rfloor \leq \frac{(n-2)! + (n-4)!}{2(n-4)!}$ for all relevant values of $n$. The equation for the vertex expands to the following quadratic on $n$:
\begin{eqnarray}
c &=& \frac{1}{2}(\frac{(n-2)!}{(n-4)!} +1) \nonumber \\
&=& \frac{1}{2}(n^2 - 5n +7).
\end{eqnarray}
The greatest number of cherries a binary tree can have is $\lfloor \frac{n}{2} \rfloor$. So, in order for $c$ to reach the vertex and for the size $\text{IN}_{(n-3)}(T)$ to begin decreasing, it must be true that $\frac{1}{2}(n^2 - 5n +7) \leq \lfloor \frac{n}{2} \rfloor$. It suffices to find where $\frac{1}{2}(n^2 - 5n +7) \leq \frac{n}{2} $, which is where
\begin{equation}
\frac{1}{2}n^2 - 3n +\frac{7}{2} \leq 0. \label{eqn:CherryQuadraticInequality}
\end{equation}
The roots of the above quadratic are located at $n = 3 \pm \sqrt{2}$. So, the inequality is only true when $n \leq 3 + \sqrt{2}$. Therefore, we will consider trees for which $n = \lfloor 3 + \sqrt{2} \rfloor = 4$. However, recall that (\ref{eqn:n-3IN_clarge}) applies only to trees with 3 or more cherries. Any tree with $n \leq 5$ taxa can only have 2 cherries. So this root of (\ref{eqn:CherryQuadraticInequality}) is also not significant.

So, the number of cherries on a given tree, $T$, with $n$ taxa and $c \geq 3$ will never reach the vertex of (\ref{eqn:n-3IN_clarge_quadratic}). Therefore, of all the trees in the set of trees with $c \geq 3$, the trees with the largest $(n-3)$-IN are those with $\lfloor \frac{n}{2} \rfloor$ cherries.
\end{proof}
The following two lemmas will be necessary to compare the size of $\text{IN}_{(n-3)}(T)$ for trees with $c \geq 3$ to that of caterpillar trees.

\begin{lemma}
The inequality,
\begin{equation}
3ub(n-1) > (n-1)! \label{eqn:ApproxInequality}
\end{equation}
holds for all $n \geq 10$. 
\end{lemma}
\begin{proof}
Base Case: $n = 10$. When $n=10$, $3ub(n-1) = 3(13)!! = 405405$ and $(n-1)! = 362880$. So (\ref{eqn:ApproxInequality}) holds for $n=10$.

Let (\ref{eqn:ApproxInequality}) hold for $n = x-1$. Then $3ub(n-2)=3(2x-9)!! > (x-2)!$. We can expand $3ub(n-1)=3(2x-7)!!$ to $3(2x-7)(2x-9)!!$ and $(x-1)!$ to $(x-1)(x-2)!$. Since $3(2x-9)!! > (x-2)!$ is true by the induction hypothesis, and $(2x-7) > (x-1)$ for all $x>6$, $3(2x-7)(2x-9)!! > (x-1)(x-2)!$ is true for all relevant $x$. Therefore, (\ref{eqn:ApproxInequality}) is true for all $n \geq 10$. 
\end{proof}

\begin{lemma}
For all natural numbers, $n$,
\begin{equation}
n^3 - 6n^2 + 11n - 6 > \frac{n^3}{2}-5n^2 + \frac{37n}{2} - 34.
\end{equation}
\end{lemma}
\begin{proof}
Let $f(n) = n^3 - 6n^2 + 11n - 6$ and $g(n) = \frac{n^3}{2}-5n^2 + \frac{37n}{2} - 34$. The difference, $f(n)-g(n)$ has two complex roots and one negative root. Observe that $f(1) = 0$ and $g(1) = -20$. Since $f(n)$ and $g(n)$ are continuous and never cross in the positive reals, and since we can find a positive value of $n$ for which $f(n)>g(n)$, this inequality holds for all positive $n$.
\end{proof}

\begin{theorem}
Of all $n$-taxa trees, the caterpillar tree has the largest $(n-3)$-interval neighborhood.
\end{theorem}

\begin{proof}
We know that (\ref{eqn:ApproxInequality}) is true for $n \geq 10$. We note that $3ub(n-1) = 4ub(n-1) - ub(n-1) < 4ub(n-1) - 2ub(n-2)$ for all $n \geq 5$. So we can replace the left hand side of (\ref{eqn:ApproxInequality}) with $4ub(n-1) - 2ub(n-2)$. The expression, $(n-1)!$, expands to $ (n-1)(n-2)(n-3)(n-4)!$. Since $(n-1)(n-2)(n-3)$ expands to $n^3 - 6n^2 + 11n - 6$ which is greater than $\frac{n^3}{2}-5n^2 + \frac{37n}{2} - 34$ for all positive real numbers, we replace the right hand side of \ref{eqn:ApproxInequality} with $(\frac{n^3}{2} - 5n^2 + \frac{37n}{2} -34)(n-4)!$. Therefore, 
 \begin{equation} \label{eqn:n-3-IN_niceInequality}
4ub(n-1) - 2ub(n-2) > (\frac{n^3}{2} - 5n^2 + \frac{37n}{2} -34)(n-4)! 
\end{equation}
 holds. After some algebraic manipulation, we found that (\ref{eqn:n-3-IN_niceInequality}) is equivalent to the following inequalities:
 \begin{eqnarray*} 
4ub(n-1) - 2ub(n-2) &>& (\frac{n^3}{2} - \frac{5n^2}{2} + 6 - \frac{n^2}{2} + \frac{n}{2} - 2n^2 + 18n - 40)(n-4)! \\
& > & (\frac{n}{2}(n^2-5n+6)-\frac{n}{2}(\frac{n}{2}-1))(n-4)! -(2(n^2 - 5n+6)+4-8n+24)(n-4)!
\end{eqnarray*}
When we simplify the above expression, we can see that the following hold as well.
\begin{eqnarray}
(2(n^2 - 5n+6)+4-8n+24)(n-4)! + 4ub(n-1) - 2ub(n-2) & > & (\frac{n}{2}(n^2-5n+6)-\frac{n}{2}(\frac{n}{2}-1))(n-4)! \nonumber \\
2(n-2)! + 4(n-4)! - 8(n-3)! + 4ub(n-1) - 2ub(n-2) &>& \frac{n}{2}(n-2)! - \frac{n}{2}(\frac{n}{2} -1)(n-4)!. \label{eqn:n-3-IN_cInequality}
\end{eqnarray}

From Theorem 5, we know that the left hand side of this equality is the size of the $(n-3)$-interval neighborhood of a caterpillar tree, while the right hand side is that of a tree with $\lfloor \frac{n}{2} \rfloor$ cherries. Therefore, the $(n-3)$-interval neighborhood of a caterpillar tree is larger than that of a tree with $\lfloor \frac{n}{2} \rfloor$ cherries for all $n \geq 10$. We can also compute that this is true for $n$ between 6 and 9.
 \end{proof}
\subsection{The $(n-3)$-Interval Neighborhood as a Proportion of All Trees}

In phylogenetics, we use simulations to test how well our methods work on "perfect" data which satisfy all of the assumptions of the given method. For instance, we can use a phylogenetics method to reconstruct a tree, and then compare it to the original tree to see how well the method worked \cite{huelsenbeck1995performance}. Distance metrics are useful for conducting these comparisons, but it is important for us to understand exactly what these metrics tell us about the trees we are comparing. So we must understand the growth of the $(n-3)$-interval neighborhood as a proportion of all trees in order to understand how $k$-IC can help us in analyzing the results of random tree simulations.

In \protect\cite{bryant2009computing}, Bryant and Steel characterized the growth of the neighborhoods of trees a given Robison-Foulds Distance away from any given tree. The Robinson-Foulds distance between $T_1$ and $T_2$ is defined to be the normalized number of splits induced by $T_1$ and not by $T_2$, and vice versa \protect\cite{robinson1981comparison}. Bryant and Steel showed that, for the Robinson-Foulds distance, the proportion of trees that share $s$ splits follows a Poisson distribution with mean $\lambda = \frac{c}{2n}$. For a given tree $T$, the expected proportion of trees that share $s$ splits with $T$ is given by $\lambda^se^{-\lambda/s}$ \protect\cite{bryant2009computing}.  When two trees, $T_1$ and $T_2$ have $s=0$, this means that they share no splits, and that they are the farthest possible Robinson-Foulds distance apart. The expected value of the proportion of trees which share 0 splits is given by

\begin{equation}
E(s=0) = e^{-c/2n}. \label{eqn:RF_Proportion}
\end{equation}

In order to understand how the proportion of trees with $s=0$ with $T$ grows for $n$ large, we can take the limit of (\ref{eqn:RF_Proportion}) as $n$ approaches infinity:
\begin{equation*}
\lim_{n \rightarrow \infty} e^{-c/2n} = 1.
\end{equation*}

This is true for any fixed $c$. So, if we fix a number of cherries on a tree, and let $n$ approach infinity, it becomes increasingly likely that we will randomly choose a tree that is the farthest possible Robinson-Foulds distance away from $T$. However, if we let $c$ be the maximum number of cherries possible, we have
\begin{equation*}
\lim_{n \rightarrow \infty} e^{-\lfloor n/2 \rfloor / 2n} = e^{-1/4} \approx 0.7788.
\end{equation*}

So, the proportion of trees the farthest possible Robinson-Foulds distance away from a given tree approaches 1 as $n$ approaches infinity for trees with a fixed number of cherries. This proportion also approaches 1 as $n$ approaches infinity and $c$ approaches 2.  This makes the Robinson-Foulds distance less useful for simulations on random large trees, since almost every resulting tree will be the farthest possible distance away from the comparison tree. This does not tell us all of the possible information about how the phylogenetic method performed. 

\begin{theorem}
The proportion of trees in the $(n-3)$-IN of a given tree approaches 0 as $n$ approaches infinity.
\end{theorem}
We used Mathematica 9.0 \protect\cite{mathematica} to compute that
\begin{equation*}
lim_{n \rightarrow \infty} \frac{2((n-2)!+2((n-4)! - 2(n-3)!) +2(2n-7)!! - (2n-9)!!)}{(2n-5)!!} = 0
\end{equation*}
and that
\begin{equation*}
lim_{n \rightarrow \infty} \frac{\lfloor n/2 \rfloor((n-2)! - (\lfloor n/2 \rfloor -1)(n-4)!}{(2n-5)!!} = 0.
\end{equation*}
These two cases suffice to show that the proportion of trees in the $(n-3)$-interval neighborhood approaches 0 as $n$ approaches infinity for all trees, since these are the largest two cases. The fact that as $n$ grows, we become less and less likely to choose a tree that is the farthest distance apart from a given tree may prove more helpful for analyzing simulations on random trees.

\section{Conclusion}

When phylogeneticists work to reconstruct the "Tree of Life", it is unlikely that every species relationship will be perfectly accurate. Instead, it is more important that relationships between larger taxonomical units be correct, and that the tree have a high degree of global accuracy. Consider $C_{3x}$ and $D_{3x}$ from Fig.~\ref{fig:CounterTrees}. For large values of $x$, the Robinson-Foulds distance will be quite large, while $C_{3x}$ and $D_{3x}$ will still satisfy 2-IC. These large trees would be assigned a worse "score" based on the RF-distance. The fact that $d(C_{3x},D_{3x}) = 2$, however, tells us that these trees are actually very similar to one another. While they are not be a perfect match, we would likely want their scores to be very good, since their topologies do match, and they exhibit large degree of global congruence. The Robinson-Foulds distance does not facilitate this as well as $k$-IC.

If we performed a method on simulated data, and returned a tree with many errors, those errors could all be located at the leaves of a tree with a very similar topology to the original. We may not necessarily want to discard this method, since it did accurately reconstruct the tree topology; however, using the Robinson-Foulds distance may lead to a low "score" on this tree, as it is still almost guaranteed to be the farthest distance away from the original for large trees. In the host-parasite trees from Figure 4 of \protect\cite{hughes2007multiple}, for example, the trees appear similar through visual comparison, and satisfy 5-interval cospeciation, but their Robinson-Foulds distance is 20. This relatively high Robinson-Foulds distance does not convey how similar the two trees actually are.

We believe that $k$-interval cospeciation also has applications to research in the coalescent model of evolution. Papers concerning the coalescent model such as \protect\cite{maddison1997gene} and \protect\cite{rosenberg2003shapes} often discuss the probability of coalescences of lineages and of congruences between gene and species trees. Knowing these probabilities, researchers in this field may be able to determine a likely $k$ between the gene and species trees in order to help them find the correct trees.

\begin{openproblem*}
Given a coalescent model in which $T$ represents the species tree and $T_i$,$T_j$ represent gene trees, find the expected values of $d(T_i,T_j)$ and $d(T,T_i)$.
\end{openproblem*}

The field of cophylogenetics is relatively new, and we require new tools to fully understand it. There are still many questions about $k$-interval cospeciation to be answered (namely, what can we say about the neighborhood of trees that between 1 and $(n-3)$-IC?). However, our research indicates that $k$-interval cospeciation is a useful and unique distance metric. Since $k$-interval cospeciation can quantify the degree of global similarity between trees, while allowing for local differences around the leaves, it will prove helpful for those who seek to test and compare cophylogenetic trees.

\section*{Acknowledgement}
This material is based upon work supported by the National Science Foundation under Grant No. DMS-1358534. Research reported in this publication was supported by an Institutional Development Award (IDeA) from the National Institute of General Medical Sciences of the National Institutes of Health under grant number P20GM12345.

\appendices

\ifCLASSOPTIONcompsoc
\else
 
\fi

\ifCLASSOPTIONcaptionsoff
  \newpage
\fi


\bibliographystyle{IEEEtran}
\bibliography{REUArticle_Bib}

\begin{thebibliography}{10}
\providecommand{\url}[1]{#1}
\csname url@samestyle\endcsname
\providecommand{\newblock}{\relax}
\providecommand{\bibinfo}[2]{#2}
\providecommand{\BIBentrySTDinterwordspacing}{\spaceskip=0pt\relax}
\providecommand{\BIBentryALTinterwordstretchfactor}{4}
\providecommand{\BIBentryALTinterwordspacing}{\spaceskip=\fontdimen2\font plus
\BIBentryALTinterwordstretchfactor\fontdimen3\font minus
  \fontdimen4\font\relax}
\providecommand{\BIBforeignlanguage}[2]{{%
\expandafter\ifx\csname l@#1\endcsname\relax
\typeout{** WARNING: IEEEtran.bst: No hyphenation pattern has been}%
\typeout{** loaded for the language `#1'. Using the pattern for}%
\typeout{** the default language instead.}%
\else
\language=\csname l@#1\endcsname
\fi
#2}}
\providecommand{\BIBdecl}{\relax}
\BIBdecl

\bibitem{huggins2012first}
P.~Huggins, M.~Owen, and R.~Yoshida, ``First steps toward the geometry of
  cophylogeny,'' in \emph{Harmony of Gr{\"o}bner Bases and the Modern
  Industrial Society}.

\bibitem{hughes2007multiple}
J.~Hughes, M.~Kennedy, K.~P. Johnson, R.~L. Palma, and R.~D. Page, ``Multiple
  cophylogenetic analyses reveal frequent cospeciation between pelecaniform
  birds and pectinopygus lice,'' \emph{Systematic biology}, vol.~56, no.~2, pp.
  232--251, 2007.

\bibitem{refregier2008cophylogeny}
G.~Refr{\'e}gier, M.~Le~Gac, F.~Jabbour, A.~Widmer, J.~A. Shykoff, R.~Yockteng,
  M.~E. Hood, and T.~Giraud, ``Cophylogeny of the anther smut fungi and their
  caryophyllaceous hosts: prevalence of host shifts and importance of
  delimiting parasite species for inferring cospeciation,'' \emph{BMC
  Evolutionary Biology}, vol.~8, no.~1, p. 100, 2008.

\bibitem{hafner1988phylogenetic}
M.~S. Hafner and S.~A. Nadler, ``Phylogenetic trees support the coevolution of
  parasites and their hosts,'' 1988.

\bibitem{maddison1997gene}
W.~P. Maddison, ``Gene trees in species trees,'' \emph{Systematic biology},
  vol.~46, no.~3, pp. 523--536, 1997.

\bibitem{li1996nearest}
M.~Li, J.~Tromp, and L.~Zhang, ``On the nearest neighbour interchange distance
  between evolutionary trees,'' \emph{Journal of Theoretical Biology}, vol.
  182, no.~4, pp. 463--467, 1996.

\bibitem{skiena1990dijkstra}
S.~Skiena, ``Dijkstra's algorithm,'' \emph{Implementing Discrete Mathematics:
  Combinatorics and Graph Theory with Mathematica, Reading, MA:
  Addison-Wesley}, pp. 225--227, 1990.

\bibitem{snir2010quartets}
S.~Snir and S.~Rao, ``Quartets maxcut: a divide and conquer quartets
  algorithm,'' \emph{IEEE/ACM Transactions on Computational Biology and
  Bioinformatics (TCBB)}, vol.~7, no.~4, pp. 704--718, 2010.

\bibitem{robinson1971comparison}
D.~F. Robinson, ``Comparison of labeled trees with valency three,''
  \emph{Journal of Combinatorial Theory, Series B}, vol.~11, no.~2, pp.
  105--119, 1971.

\bibitem{steel2014tracing}
M.~Steel, ``Tracing evolutionary links between species,'' \emph{arXiv preprint
  arXiv:1402.3771}, 2014.

\bibitem{huelsenbeck1995performance}
J.~P. Huelsenbeck, ``Performance of phylogenetic methods in simulation,''
  \emph{Systematic biology}, vol.~44, no.~1, pp. 17--48, 1995.

\bibitem{bryant2009computing}
D.~Bryant and M.~Steel, ``Computing the distribution of a tree metric,''
  \emph{IEEE/ACM Transactions on Computational Biology and Bioinformatics
  (TCBB)}, vol.~6, no.~3, pp. 420--426, 2009.

\bibitem{robinson1981comparison}
D.~Robinson and L.~R. Foulds, ``Comparison of phylogenetic trees,''
  \emph{Mathematical Biosciences}, vol.~53, no.~1, pp. 131--147, 1981.

\bibitem{mathematica}
W.~R. Inc., ``Mathematica,'' 2012.

\bibitem{rosenberg2003shapes}
N.~A. Rosenberg, ``The shapes of neutral gene genealogies in two species:
  probabilities of monophyly, paraphyly, and polyphyly in a coalescent model,''
  \emph{Evolution}, vol.~57, no.~7, pp. 1465--1477, 2003.

\end{thebibliography}
\end{document}